\documentclass[12pt,a4paper]{article}

\usepackage{amsmath,amscd,amssymb,amsthm,mathrsfs,amsfonts,layout,indentfirst,graphicx,caption,mathabx,  tikz, stmaryrd,appendix,calc,imakeidx,upgreek,tipa}
\usepackage{palatino}  
\usepackage{slashed} 
\usepackage{mathrsfs} 
\usepackage{extarrows} 
\makeindex

\usepackage{lipsum}
\let\OLDthebibliography\thebibliography
\renewcommand\thebibliography[1]{
	\OLDthebibliography{#1}
	\setlength{\parskip}{0pt}
	\setlength{\itemsep}{2pt} 
}

\allowdisplaybreaks  
\usepackage{latexsym}
\usepackage{chngcntr}
\usepackage[colorlinks,linkcolor=blue,anchorcolor=blue,linktocpage]{hyperref}
\hypersetup{citecolor=[rgb]{0,0.5,0}}

\usepackage{fullpage}
\counterwithin{figure}{section}

\theoremstyle{definition}
\newtheorem{df}{Definition}[section]

\newtheorem{rem}[df]{Remark}

\theoremstyle{plain}
\newtheorem{thm}[df]{Theorem}

\newtheorem{pp}[df]{Proposition}
\newtheorem{co}[df]{Corollary}

\newcommand{\fk}{\mathfrak}
\newcommand{\mc}{\mathcal}

\newcommand{\wch}{\widecheck}

\newcommand{\ovl}{\overline}
\newcommand{\tr}{\mathrm{t}} 

\newcommand{\End}{\mathrm{End}} 
\newcommand{\id}{\mathbf{1}}
\newcommand{\Hom}{\mathrm{Hom}}

\newcommand{\ev}{\mathrm{ev}}
\newcommand{\coev}{\mathrm{coev}}

\newcommand{\Rep}{\mathrm{Rep}}

\newcommand{\bk}[1]{\langle {#1}\rangle}

\newcommand{\scr}{\mathscr}

\newcommand{\gk}{\mathfrak g}

\newcommand{\im}{\mathbf{i}}
\newcommand{\Co}{\complement}

\newcommand{\mbb}{\mathbb}

\newcommand{\blt}{\bullet}

\newcommand{\Cbb}{\mathbb C}
\newcommand{\Nbb}{\mathbb N}
\newcommand{\Zbb}{\mathbb Z}

\newcommand{\wt}{\mathrm{wt}}

\newcommand{\tth}{\text{\textturnh}}

\numberwithin{equation}{section}

\title{Regular vertex operator subalgebras and compressions of intertwining operators}
\author{{\sc Bin Gui}
}
\date{}
\begin{document}\sloppy 
	\pagenumbering{arabic}
	\setcounter{section}{-1}

	\maketitle

\begin{abstract}
Let $V$ be a vertex operator subalgebra of $U$. Assume that $U$, $V$, and its commutant  $V^c$ in $U$ are CFT-type, self-dual, and regular VOAs. Assume also that the double commutant $V^{cc}$ equals $V$. We prove that any intertwining operator of $V$ is a compression of intertwining operators of $U$.
\end{abstract}

\section{Introduction}

In \cite{KM15}, Krauel-Miyamoto showed that if $V$ is a vertex operator subalgebra of $U$, if $U$, $V$, and the commutant $V^c$ are CFT-type, self-dual, and regular VOAs, and if $V^{cc}=V$, then any irreducible $V$-module appears  in some irreducible $U$-module. For example, by \cite{Ara15a,Ara15b,ACL19}, these assumptions are satisfied when we take $V\subset U$ to be $L_{k+1}(\gk)\subset L_k(\gk)\otimes L_1(\gk)$, where $k$ is a positive integer, $\gk$ is a finite dimensional complex simple Lie algebra of type $ADE$, and $L_k(\gk)$ is the corresponding (unitary) affine VOA. In this case, $V^c$ is a discrete series principle $W$-algebra $\mc W_l(\gk)$.

The above results have important applications to the unitarity problems in VOAs: For example, we can conclude that any irreducible $\mc W_l(\gk)$-module is unitarizable since this is true for any unitary affine VOA. Moreover,  using these results (together with the techniques developed in  \cite{Ten17,Ten19a,Ten19b}), Tener showed in \cite{Ten19c} that the modular tensor categories associated to all unitary affine VOAs and type $AE$ discrete series $W$-algebras are unitary, and solved a longstanding problem in subfactor theory and algebraic quantum field theory: that the conformal nets associated to unitary affine VOAs and type $ADE$ discrete series $W$-algebras are completely rational.

In this paper, we generalize the result of \cite{KM15} to intertwining operators: We show that any  intertwining operator of $V$ is a compression of intertwining operators of $U$ (theorem \ref{lb7}). To be more precise, let $W_I,W_J,W_K$ be (ordinary) $U$-modules, which can also be regarded as weak $V$-modules. Suppose that $\mc Y^U$ is a type $W_K\choose W_IW_J$ intertwining operator of $U$, and $W_i,W_j,W_k$ are graded irreducible $V$-submodules of $W_I,W_J,W_K$ respectively, then one can find $\lambda\in\mbb Q$ such that  $z^\lambda$  times the restriction of $\mc Y^U$ to $W_i,W_j,W_k$ is an intertwining operator $\mc Y$ of $V$. (Note that without the factor $z^\lambda$, the restriction itself may not satisfy the $L_{-1}$-derivative property.) We then say that $\mc Y$ is a \textbf{compression} of $\mc Y^U$. (See definition \ref{lb8} for more details.) Our main result of this article is that any intertwining operator of $V$ can be written as a (finite) sum of those that are compressions of intertwining operators of $U$.

Our result can be applied to prove many important functional analytic properties for intertwining operators.  One such property is the (polynomial) energy bounds condition \cite{CKLW18,Gui19a,Gui19b}, which says roughly that the smeared intertwining operators are bounded by $L_0^n$ for some $n\geq0$. Proving energy bounds condition for intertwining operators is a key step in relating the tensor structures of VOA modules and the corresponding conformal net modules; see \cite{Was98,TL04, Gui18,Gui20}. On the other hand,  one may deduce the energy bounds condition of the compressed intertwining operator $\mc Y$ from that of $\mc Y^U$. Since, by our main result, any intertwining operator of $L_k(\gk)$ or $\mc W_l(\gk)$ (when $\gk$ is of type $ADE$) is a compression of tensor products of intertwining operators of $L_1(\gk)$, and since the latter were proved in \cite{TL04} to be energy bounded, we can conclude that all intertwining operators of $L_k(\gk)$ or $\mc W_l(\gk)$ are energy bounded. This result will be used in \cite{Gui20} to show that if $V$ is $L_k(\gk)$ or $\mc W_l(\gk)$, and if $\mc A$ is the corresponding conformal net, then the tensor and braid structures of the representation categories of $V$ and $\mc A$ are compatible.

\subsubsection*{Acknowledgment}

The author would like to thank James Tener for helpful discussions on the topic of this paper.

\section{Intertwining operators and tensor categories}

Let $(V,Y,\id,\nu)$ be a vertex operator algebra (VOA) where $\id$ is the vacuum vector and $\nu$ is the conformal vector. For any $v\in V$, write $Y(v,z)=\sum_{n\in\Zbb} Y(v)_nz^{-n-1}$ where $Y(v)_n\in\End(V)$. Then $L_n:=Y(\nu)_{n+1}$ satisfy the Virasoro relation
\begin{align*}
[L_n,L_m]=(n-m)L_{n+m}+\delta_{n,-m}\frac{n^3-n}{12}c,
\end{align*}
where $c$ is the central charge of $V$. We shall always assume that $V$ is CFT-type, namely, $V$ has $L_0$-grading $V=\bigoplus_{n\in\Nbb}V(n)$ and $V(0)=\Cbb\id$. We also assume that $V$ is self-dual and regular (equivalently, self-dual, rational and $C_2$-cofinite \cite{ABD04}). Note that the self-dual condition is equivalent to the existence of a non-degenerate invariant bilinear form. As a consequence of CFT-type and being self-dual, $V$ is simple. (See, for example, \cite{CKLW18} proposition 6.4-(iv).) Moreover, any (ordinary) $V$-module is semisimple, and the category of $V$-modules is a rigid modular tensor category \cite{Hua08}.

We write $V$-modules as $W_i,W_j,W_k,\dots$ whose vertex operators are denoted by $Y_i,Y_j,Y_k,\dots$ respectively. $V$ itself as a $V$-module (the vacuum module) will also be written as $W_0$. We write $Y_i(v,z)=\sum_{n\in\Zbb}Y_i(v)_nz^{-n-1}$ where each $Y_i(v)_n\in\End(W_i)$. Again, any $V$-module has $L_0$-grading  
$W_i=\bigoplus_{n\in\Cbb}W(n)$. Recall that a homogeneous vector of $W_i$ is, by definition, an eigenvector of $L_0$. In the case that $W_i$ is irreducible (i.e. simple), we furthermore have $W_i=\bigoplus_{n\in\Nbb+\alpha}W(n)$ for some $\alpha\in\Cbb$. (Indeed, $\alpha\in\mbb Q$ by \cite{AM88,DLM00}.) We let $P_n$ denote the projection of $W_i$ onto $W_i(n)$. We also let
\begin{gather*}
W_i(\leq n)=\bigoplus_{\mathrm{Re}(m)\leq n}W(m),
\end{gather*}
and let $P_{\leq n}$ be the projection of $W_i$ onto $W_i(\leq n)$.

Let $W_{\ovl i}$ denote the contragredient module of $W_i$. Recall that as a vector space, $W_{\ovl i}=\bigoplus_{n\in\Cbb}W_i(n)^*$. (See \cite{FHL93} for more details.) The evaluation between $w'\in W_{\ovl i}$ and $w\in W_i$ is written as $\bk {w,w'}$ or $\bk{w',w}$. (The same notation will be used if one of $w,w'$ is in the algebraic completion.) Since $V$ is self-dual, we identify the vacuum module $V=W_0$ and its contragredient module $W_{\ovl 0}$. $W_{\ovl{\ovl i}}$ is identified with $W_i$ in an obvious way.

Recall that if  $W_i,W_j,W_k$ are $V$-modules, an intertwining operator $\mc Y$ of type $W_k\choose W_iW_j$ (or $k\choose i~j$ for short) is a linear map
\begin{gather*}
W_i\rightarrow \End(W_j,W_k)\{z\}\\
w_i\mapsto \mc Y(w^{(i)},z)=\sum_{n\in\Cbb}\mc Y(w^{(i)})_nz^{-n-1}
\end{gather*}
where the sum above is the formal sum,  each $\mc Y(w^{(i)})_n$ is in $\End(W_j,W_k)$, and the following conditions are satisfied:

(a) (Lower truncation) For any $w^{(j)}\in W_j$, $\mathcal Y(w^{(i)})_nw^{(j)}=0$ when $\mathrm{Re}(n)$ is sufficiently large.

(b) (Jacobi identity) For any $u\in V,w^{(i)}\in W_i,m,n\in\mathbb Z,s\in\Cbb$, we have
\begin{align}
&\sum_{l\in\Nbb}{m\choose l}\mathcal Y\big(Y_i(u)_{n+l}w^{(i)}\big)_{m+s-l}\nonumber\\
=&\sum_{l\in\Nbb}(-1)^l{n\choose l}Y_k(u)_{m+n-l}\mathcal Y(w^{(i)})_{s+l}-\sum_{l\in\Nbb}(-1)^{l+n}{n\choose l}\mathcal Y(w^{(i)})_{n+s-l}Y_j(u)_{m+l}.
\end{align}

(c)	($L_{-1}$-derivative) $	\frac d{dz} \mathcal Y(w^{(i)},z)=\mathcal Y(L_{-1}w^{(i)},z)$ for any $w^{(i)}\in W_i$.\\
We say that $W_i,W_j,W_k$ are respectively the \textbf{charge space}, the \textbf{source space}, and the \textbf{target space} of $\mc Y$. If $W_i,W_j,W_k$ are all irreducible, we say that $\mc Y$ is an \textbf{irreducible intertwining operator}.

Recall that given $\mc Y\in\mc V{k\choose i~j}$, one can define $\mbb B\mc Y\in\mc V{k\choose j~i}$ and $\Co\mc Y\in\mc V{\ovl j\choose i~\ovl k}$, called the (positively) \textbf{braided} intertwining operator and the \textbf{contragredient} intertwining operator of $\mc Y$, by choosing any $w^{(i)}\in W_i,w^{(j)}\in W_j,w^{(\ovl k)}\in W_{\ovl k}$ and setting
\begin{gather*}
\mbb B\mc Y(w^{(j)},z)w^{(i)}=e^{zL_{-1}}\mc Y(w^{(i)},e^{\im\pi}z)w^{(j)},\\
\bk{\Co\mc Y(w^{(i)},z)w^{(\ovl k)},w^{(j)}}=\bk{w^{(\ovl k)}, \mc Y(e^{zL_1}(e^{\im\pi}z^{-2})^{L_0}w^{(i)},z^{-1})w^{(j)} }.
\end{gather*}
In particular, $Y_{\ovl i}\in\mc V{\ovl i\choose 0~\ovl i}$ (the vertex operator for $W_{\ovl i}$) is contragredient to $Y_i$. See \cite{FHL93} for details.

We refer the reader to \cite{HL13} and the references therein for the definition and basic properties of the tensor category $\Rep(V)$ of $V$-modules. (See also \cite{Gui19a} for a sketch of the Huang-Lepowsky tensor product theory.) Roughly speaking, $\Rep(V)$ is defined such that the fusion rules are exactly the dimensions of the spaces of intertwining operators, and the $R$- and $F$-matrices are described by the braid and the fusion relations of intertwining operators. 

To be more precise, the tensor functor (fusion product) $\boxtimes$ is defined in such a way that there is a functorial isomorphism
\begin{gather}
\Hom_V(W_i\boxtimes W_j,W_k)\xrightarrow{\simeq} \mc V{k\choose i~j},\qquad \upalpha\mapsto\mc Y_\upalpha\label{eq12}
\end{gather}
for any $V$-modules $W_i,W_j,W_k$. By saying  this map is functorial, we mean that if $F\in\Hom_V(W_{i'},W_i)$, $G\in\Hom_V(W_{j'},W_j)$, and $H\in\Hom_V(W_k,W_{k'})$, then for any $w^{(i')}\in W_{i'}$ and $w^{(j')}\in W_{j'}$,
\begin{align*}
\mc Y_{H\upalpha(F\otimes G)}(w^{(i')},z)w^{(j')}=H\mc Y_\upalpha(Fw^{(i')},z)Gw^{(j')}.
\end{align*}
In particular, we denote by $\mc L_{i,j}\in\mc V{i\boxtimes j\choose i~j}=\mc V{W_i\boxtimes W_j\choose W_i~W_j}$ the  value of the identity element $\id_{i\boxtimes j}\in\Hom_V(W_i\boxtimes W_j,W_i\boxtimes W_j)$ under \eqref{eq12}, i.e.,
\begin{align*}
\mc L_{i,j}=\mc Y_{\id_{i\boxtimes j}}.
\end{align*}
Here we adopt the notation
\begin{align*}
W_{i\boxtimes j}=W_i\boxtimes W_j.
\end{align*}
Then for any $\upalpha\in\Hom_V(W_i\boxtimes W_j,W_k)$, by the functoriality of \eqref{eq12} and that $\upalpha=\upalpha\cdot \id_{W_i\boxtimes W_j}$, it is clear that
\begin{gather*}
\mc Y_\upalpha=\upalpha\mc L_{i,j}.
\end{gather*}
(In the papers of Huang-Lepowsky, $\mc L_{i,j}(w^{(i)},z)w^{(j)}$ is written as $w^{(i)}\boxtimes_{P(z)}w^{(j)}$, regarded as a fusion product of the vectors $w^{(i)},w^{(j)}$.)

Note that $Y_i\in\mc V{i\choose 0~i}$. The left unitor $W_0\boxtimes W_i\xrightarrow{\simeq}W_i$ is defined such that it is sent by \eqref{eq12} to the element $Y_i$. The right unitor
\begin{align*}
\upkappa(i):W_i\boxtimes W_0\xrightarrow{\simeq} W_i
\end{align*}
is defined such that
\begin{align*}
\mc Y_{\upkappa(i)}=\mbb B Y_i.
\end{align*}
We call $\mc Y_{\upkappa(i)}\in\mc V{i\choose i~0}$ the \textbf{creation operator} of $W_i$.

The braid isomorphism $\mbb B=\mbb B_{i,j}:W_i\boxtimes W_j\xrightarrow{\simeq} W_j\boxtimes W_i$ is defined such that
\begin{align*}
\mc Y_{\mbb B_{i,j}}=\mbb B\mc L_{i,j}.
\end{align*}
We will not use braiding in this article. To describe the associativity isomorphisms, we first notice:
\begin{pp}\label{lb1}
Choose any $z\in\Cbb^\times=\Cbb-\{0\}$. Then for any $n\in\Cbb$,
\begin{gather}
\underset{w^{(i)}\in W_i,w^{(j)}\in W_j}{\mathrm{Span}}P_{\leq n}\cdot \mc L_{i,j}(w^{(i)},z)w^{(j)}=(W_i\boxtimes W_j)(\leq n).\label{eq1}
\end{gather}
\end{pp}
This proposition was proved in \cite{Gui19a} section A.2.\footnote{Proposition \ref{lb1} is similar to but slightly stronger than \cite{Hua95} Lemma 14.9. That lemma says that \eqref{eq1} holds with $P_{\leq s}$ replaced by $P_s$ and $(W_i\boxtimes W_j)(\leq s)$ by $(W_i\boxtimes W_j)(s)$. Huang's result is enough for applications in our paper.}  The main idea of the proof is as follows. Set $W_k=W_i\boxtimes W_j$. Then it is equivalent to proving that for any $w^{(\ovl k)}\in W_{\ovl k}$, if
\begin{align*}
\bk{w^{(\ovl k)},\mc L_{i,j}(w^{(i)},z)w^{(j)}}=0
\end{align*}
for any $w^{(i)}\in W_i$ and $w^{(j)}\in W_j$, then $w^{(\ovl k)}=0$. To see this, let $\mc W\subset W_{\ovl k}$ be the subspace of all $w^{(\ovl k)}$ satisfying the above identity. Using the Jacobi identity for intertwining operators, it is easy to see that $\mc W$ is $V$-invariant, i.e. it is a $V$-submodule of $W_{\ovl k}$. Assume $\mc W$ is non-trivial. Choose an irreducible module $W_l$ such that $\mc W$ has an irreducible submodule isomorphic to $W_{\ovl l}$. Then there is a non-zero $T\in\Hom_V(W_k,W_l)$ whose transpose $T^\tr\in \Hom_V(W_{\ovl l},W_{\ovl k})$ maps $W_{\ovl l}$ into $\mc W$. Using the definition of $\mc W$ and the fact that $T^\tr w^{(\ovl l)}\in\mc W$ for each $w^{(\ovl l)}\in W_{\ovl l}$, it is easy to see that $T\mc L_{i,j}=0$. Recall $W_k=W_i\boxtimes W_j$. So we can write $T\mc L_{i,j}=\mc Y_T$. Therefore $T=0$, which gives a contradiction.

\begin{co}\label{lb2}
Let $W_i,W_j,W_s$ be $V$-modules, and assume that $W_s$ is isomorphic to an  irreducible submodule of $W_i\boxtimes W_j$. If $\Xi^s_{i,j}$ is a basis of $\Hom_V(W_i\boxtimes W_j,W_s)$, $\upalpha\in\Xi^s_{i,j}$,  then  there exist homogeneous vectors $w^{(i)}_1,\dots,w^{(i)}_m\in W_i$,  $w^{(j)}_1,\dots,w^{(j)}_m\in W_j$, $w^{(\ovl s)}\in W_{\ovl s}$, and constants $\lambda_1,\dots,\lambda_m\in\mbb Q$, such that for any $\upbeta\in\Xi^s_{i,j}$, the expression
\begin{align}
\sum_{l=1}^m z^{\lambda_l}\bk{\mc Y_{\upbeta}(w^{(i)}_l,z)w^{(j)}_l,w^{(\ovl s)}}\label{eq15}
\end{align}
is a constant (where $z$ is a complex variable), and it is non-zero if and only if $\upbeta=\upalpha$.
\end{co}
\begin{proof}
Choose $n\in\Cbb$ such that $W_s(n)$ is non-trivial. Choose for each $\upalpha\in\Xi^s_{i,j}$ a morphism $\wch\upalpha\in\Hom_V(W_s,W_i\boxtimes W_j)$ such that $\upalpha\wch\upbeta=\delta_{\upalpha,\upbeta}\id_s$ for any $\upalpha,\upbeta\in\Xi^s_{i,j}$. Now we fix $\upalpha\in\Xi^s_{i,j}$, and choose a non-zero vector $w^{(s)}\in W_s(n)$. By proposition \ref{lb1}, there exist homogeneous vectors $w^{(i)}_1,\dots,w^{(i)}_m\in W_i$ and $w^{(j)}_1,\dots,w^{(j)}_m\in W_j$ such that
\begin{align*}
\sum_{l=1}^m P_n\mc L_{i,j}(w^{(i)}_l,1)w^{(j)}_l=\wch\upalpha w^{(s)}.
\end{align*}
Choose $w^{(\ovl s)}\in W_{\ovl s}(n)$ such that $\bk{w^{(s)},w^{(\ovl s)}}\neq 0$. (Recall that $W_{\ovl s}(n)$ is the dual vector space of $W_s(n)$.) Then it is easy to check that the constant
\begin{align}
\sum_{l=1}^m \bk{\mc Y_{\upbeta}(w^{(i)}_l,1)w^{(j)}_l,w^{(\ovl s)}}\label{eq14}
\end{align}
is non-zero (in which case equals $\bk{w^{(s)},w^{(\ovl s)}}$) if and only if $\upbeta=\upalpha$. Now let
\begin{align*}
\lambda_l=\wt(w_l^{(s)})-\wt(w_l^{(i)})-\wt(w_l^{(j)}),
\end{align*}
where the three terms on the right hand side are the  weights (i.e. the $L_0$-eigenvalues) of the corresponding homogeneous vectors. Then \eqref{eq14} equals \eqref{eq15}. This proves the claim of this corollary.
\end{proof}

For any irreducible equivalence class of $V$-modules we choose a representing element and let them form a (finite) set $\mc E$. We assume $W_0\in\mc E$. Choose $V$-modules $W_i,W_j,W_k,W_l$. If $W_s\in\mc E$, we write $s\in\mc E$ for short. For each $r\in\mc E$, we choose  bases $\Xi_{j,k}^r$ of $\Hom_V(W_j\boxtimes W_k,W_r)$ and $\Xi_{i,r}^l$ of $\Hom_V(W_i\boxtimes W_r,W_l)$. Choose $z,\zeta\in\Cbb^\times$ with $0<|\zeta|<|z|$. Then for any $w^{(i)}\in W_i,w^{(j)}\in W_j$, and any $\upalpha\in\Xi^l_{i,r},\upbeta\in\Xi^r_{j,k}$ the product
\begin{align}
\mc Y_\upalpha(w^{(i)},z)\mc Y_\upbeta(w^{(j)},\zeta)\label{eq28}
\end{align}
converges absolutely \cite{Hua05}, in the sense that for any $w^{(k)}\in W_k,w^{(\ovl l)}\in W_{\ovl l}$, 
\begin{align*}
\sum_{n\in\Cbb}\big|\bk{\mc Y_\upalpha(w^{(i)},z)P_n\mc Y_\upbeta(w^{(j)},\zeta)w^{(k)},w^{(\ovl l)}}\big|<+\infty.
\end{align*}
Moreover, consider the expression
\begin{align}
\bk{\mc Y_\upalpha(w^{(i)},z)\mc Y_\upbeta(w^{(j)},\zeta)w^{(k)},w^{(\ovl l)}}\label{eq2}
\end{align}
as an element of $(W_i\otimes W_j\otimes W_k\otimes W_{\ovl l})^*$. (Note that this element depends also on the arguments $\arg z$ and $\arg \zeta$.) By linearity, we have a linear map
\begin{align}
\Psi_{z,\zeta}:\bigoplus_{r\in\mc E}\Hom_V(W_i\boxtimes W_r,W_l)\otimes \Hom_V(W_j\boxtimes W_k,W_r)\rightarrow(W_i\otimes W_j\otimes W_k\otimes W_{\ovl l})^*
\end{align}
sending $\upalpha\otimes\upbeta$ to the linear functional defined by \eqref{eq2}. This map is well-known to be injective. Indeed,  choose any $\fk X$ in the domain of $\Psi_{z,\zeta}$. If $\Psi_{z,\zeta}(\fk X)$ equals $0$ for one pair $(z,\zeta)$, then, by the existence of differential equations as in \cite{Hua05}, $\Psi_{z,\zeta}(\fk X)$ equals $0$ for all $z,\zeta$ satisfying $0<|\zeta|<|z|$. (See \cite{Gui19a} the paragraphs after theorem 2.4 for a detailed explanation.) Then, using proposition \ref{lb1}, it is not hard to show $\fk X=0$. (See for instance \cite{Gui19a} proposition 2.3 whose proof is in section A.2.)

Similarly, when $0<|z-\zeta|<|\zeta|$, one can define an injective linear map
\begin{align}
\Phi_{z,\zeta}:\bigoplus_{s\in\mc E}\Hom_V(W_s\boxtimes W_k,W_l)\otimes \Hom_V(W_i\boxtimes W_j,W_s)\rightarrow(W_i\otimes W_j\otimes W_k\otimes W_{\ovl l})^*\label{eq22}
\end{align}
sending each $\upgamma\otimes\updelta$ to the linear functional determined by the following iterate of intertwining operators:
\begin{align}
\bk{\mc Y_\upgamma(\mc Y_\updelta(w^{(i)},z-\zeta)w^{(j)},\zeta)w^{(k)},w^{(\ovl l)}}.
\end{align}
Again, this map depends on the choice of arguments: $\arg(z-\zeta)$ and $\arg(\zeta)$, and the above expression converges absolutely in an appropriate sense. By a deep result of \cite{Hua95,Hua05}, $\Phi_{z,\zeta}$ and $\Psi_{z,\zeta}$ have the same image. 

We now assume
\begin{gather}
0<|z-\zeta|<|\zeta|<|z|,\qquad \arg(z-\zeta)=\arg\zeta=\arg z.\label{eq8}
\end{gather}
In particular, $\zeta,z$ are on the same ray starting from the origin. If we also choose bases $\Xi^l_{s,k}$ and $\Xi^s_{i,j}$ of $\Hom_V(W_s\boxtimes W_k,W_l)$ and $\Hom_V(W_i\boxtimes W_j,W_s)$ respectively, then we have a matrix  $\{F_{\upgamma\updelta}^{\upalpha\upbeta}\}$ (the fusion matrix) representing the invertible map $\Psi_{z,\zeta}\Phi_{z,\zeta}^{-1}$. Equivalently, we have a unique number $F_{\upgamma\updelta}^{\upalpha\upbeta}$ for each $\upalpha,\upbeta,\upgamma,\updelta$ such that for each $s\in\mc E$ and each $\upgamma\in\Xi^l_{s,k},\updelta\in\Xi^s_{i,j}$, the fusion relation
\begin{align}
\mc Y_\upgamma(\mc Y_\updelta(w^{(i)},z-\zeta)w^{(j)},\zeta)=\sum_{r\in\mc E}\sum_{\upalpha\in\Xi^l_{i,r},\upbeta\in\Xi^r_{j,k}}F_{\upgamma\updelta}^{\upalpha\upbeta}\cdot\mc Y_\upalpha(w^{(i)},z)\mc Y_\upbeta(w^{(j)},\zeta)\label{eq6}
\end{align} 
holds for each $w^{(i)}\in W_i,w^{(j)}\in W_j$. This fusion matrix is independent of the particular choice of $z,\zeta$ satisfying the above mentioned conditions. The associativity isomorphisms of $\Rep(V)$ are defined in such a way that after making $\Rep(V)$ strict, we have
\begin{align}
\upgamma(\updelta\otimes\id_k)=\sum_{r\in\mc E}\sum_{\upalpha\in\Xi^l_{i,r},\upbeta\in\Xi^r_{j,k}}F_{\upgamma\updelta}^{\upalpha\upbeta}\cdot\upalpha(\id_i\otimes\upbeta),\label{eq7}
\end{align}
namely, $F$ is also an $F$-matrix of $\Rep(V)$.

\section{Fusion of annihiliation and vertex operators}\label{lb3}

Let $W_i,W_j$ be $V$-modules. For each $W_s\in\mc E$,  there is a non-degenerate bilinear form $\bk{\cdot,\cdot}$ on $\Hom_V(W_i\boxtimes W_j,W_s)\otimes \Hom_V(W_s,W_i\boxtimes W_j)$ such that if $\upalpha\in\Hom_V(W_i\boxtimes W_j,W_s)$ and $T\in \Hom_V(W_s,W_i\boxtimes W_j)$, then
\begin{align}
\upalpha T=\bk{\upalpha,T}\id_s.
\end{align}
This bilinear form gives an isomorphism
\begin{align}
\Hom_V(W_s,W_i\boxtimes W_j)\xrightarrow{\simeq}\Hom_V(W_i\boxtimes W_j,W_s)^*.\label{eq5}
\end{align}
We shall always identify $\Hom_V(W_s,W_i\boxtimes W_j)$ and $\Hom_V(W_i\boxtimes W_j,W_s)^*$ using the above isomorphism.

Recall from the last section that $\Xi^s_{i,j}$ is a basis of $\Hom_V(W_i\boxtimes W_j,W_s)$. Then we can choose a dual basis $\{\wch\upalpha:\upalpha\in\Xi^s_{i,j}\}$. Namely, for each $\upalpha\in\Xi^s_{i,j}$, we have $\wch\upalpha\in\Hom_V(W_s,W_i\boxtimes W_j)$, and if $\upbeta\in\Xi^s_{i,j}$, then $\bk{\upalpha,\wch\upbeta}=\delta_{\upalpha,\upbeta}$. So we also have
\begin{align}
\upalpha\wch\upbeta=\delta_{\upalpha,\upbeta}\id_s.\label{eq3}
\end{align}
This implies that
\begin{align}
\id_{i\boxtimes j}=\sum_{s\in\mc E}\sum_{\upalpha\in\Xi^s_{i,j}}\wch\upalpha\upalpha,\label{eq4}
\end{align}
since, by \eqref{eq3}, the left multiplications of both sides of \eqref{eq4} by any $\upbeta\in\Xi^s_{i,j}$ equal $\upbeta$.

In \cite{HK07}, Huang-Kong used the rigidity of $\Rep(V)$ to define a natural isomorphism $\Hom_V(W_{\ovl i}\boxtimes W_{\ovl j},W_{\ovl s})\xrightarrow\simeq \Hom_V(W_i\boxtimes W_j,W_s)^*$. Since $\mc V{\ovl s\choose \ovl i~\ovl j}$ is isomorphic to $\mc V{j\choose\ovl i~s}$ by sending $\mc Y$ to $\Co\mc Y$, we also have an isomorphism
\begin{gather}
\tth:\Hom_V(W_i\boxtimes W_j,W_s)^*\xrightarrow\simeq \Hom_V(W_{\ovl i}\boxtimes W_s,W_j).
\end{gather}
In the following, we review the construction of this isomorphism.

In \cite{Hua08}, Huang showed  that $\Rep(V)$ is rigid, and the (categorical) dual object of any $V$-module $W_i$ could be chosen to be the contragredient module $W_{\ovl i}$. Moreover, if we define
\begin{align*}
\ev_{\ovl i,i}\in\Hom_V(W_{\ovl i}\boxtimes W_i,V)
\end{align*}
such that
\begin{align*}
\mc Y_{\ev_{\ovl i,i}}=\Co\mc Y_{\upkappa(\ovl i)},
\end{align*}
(Recall that $\mc Y_{\upkappa(\ovl i)}$ is the creation operator of $W_{\ovl i}$, which is of type $\ovl i\choose \ovl i~0$. $\mc Y_{\ev_{\ovl i,i}}$, which is of type $0\choose \ovl i~i$, is called the \textbf{annihilation operator} of $W_i$.) then there is a (unique) morphism
\begin{align*}
\coev_{i,\ovl i}\in\Hom_V(V,W_i\boxtimes W_{\ovl i})
\end{align*}
satisfying the conjugate equations
\begin{gather*}
(\id_i\otimes\ev_{\ovl i,i})(\coev_{i,\ovl i}\otimes\id_i)=\id_i,\\
(\ev_{\ovl i,i}\otimes\id_{\ovl i})(\id_{\ovl i}\otimes\coev_{i,\ovl i})=\id_{\ovl i}.
\end{gather*}
This is also true for $W_{\ovl i}$. Thus we have $\ev_{i,\ovl i}$ and $\coev_{\ovl i,i}$ defined by $\ev_{i,\ovl i}=\ev_{\ovl{\ovl i},\ovl i}$ and $\coev_{\ovl i,i}=\coev_{\ovl i,\ovl{\ovl i}}$. 

Recall the identification \eqref{eq5}. We define
\begin{gather}
\tth:\Hom_V(W_s,W_i\boxtimes W_j)\xrightarrow\simeq \Hom_V(W_{\ovl i}\boxtimes W_s,W_j),\nonumber\\
T\mapsto \tth(T)=(\ev_{\ovl i,i}\otimes\id_j)(\id_{\ovl i}\otimes T).
\end{gather}
That $\tth$ is an isomorphism follows from the conjugate equations. Using the definition of $\tth$ and equation \eqref{eq4}, it is easy to see
\begin{align}
\ev_{\ovl i,i}\otimes\id_j=\sum_{s\in\mc E}\sum_{\upalpha\in\Xi^s_{i,j}}\tth(\wch\upalpha)(\id_{\ovl i}\otimes\upalpha).
\end{align}
Thus, by \eqref{eq6} and \eqref{eq7}, we have the following fusion relation which will play an important role in later sections: Let $z,\zeta\in\Cbb$ satisfy \eqref{eq8}. Then for any $w^{(i)}\in W_i,w^{(\ovl i)}\in W_{\ovl i}$,
\begin{align}
Y_j(\mc Y_{\ev_{\ovl i,i}}(w^{(\ovl i)},z-\zeta)w^{(i)},\zeta)=\sum_{s\in\mc E}\sum_{\upalpha\in\Xi^s_{i,j}}\mc Y_{\tth(\wch\upalpha)}(w^{(\ovl i)},z)\mc Y_\upalpha(w^{(i)},\zeta).\label{eq9}
\end{align}
Note that for each $s\in\mc E$, $\{\tth(\wch\upalpha):\upalpha\in\Xi^s_{i,j}\}$ is a basis of $\Hom_V(W_{\ovl i}\boxtimes W_s,W_j)$. Roughly speaking, this fusion relation says that any intertwining operator arises from fusing the annihilation operators and the vertex operators. This is parallel to the fact that any $V$-module character occurs in the sum resulting from the modular transformation $\tau\mapsto -1/\tau$ of the vacuum module character.

\section{Compressions of intertwining operators}

Assume that $V$ is a vertex operator subalgebra (sub-VOA for short) of another CFT-type VOA $U$ with vertex operator $Y^U$ and conformal vector $\omega$. This means that $V$ is a subspace of $U$, $V$ and $U$ share the same vacuum vector $\id$, and that $Y^U(v_1,z)v_2=Y(v_1,z)v_2$ when $v_1,v_2\in V$. Let $L^U_n=Y^U(\omega)_{n+1}$. We shall always assume the additional condition that
\begin{align}
L^U_0\nu=2\nu,\qquad L^U_1\nu=0.\label{eq16}
\end{align}
Then by \cite{FZ92} or \cite{LL12} theorem 3.11.12,  $(V^c,Y',\id,\nu')$ is a sub-VOA of $U$, where $V^c$ is the set of all $u\in U$ such that $Y(v)_nu=0$ for all $v\in V$ and $n\in\Nbb$, $Y'$ is the restriction of $Y^U$ to $V^c$, and $\nu'=\omega-\nu$. We set $L'_n=Y'(\nu')_{n+1}$. 

Assume that $V^c$ is self-dual, CFT-type, and regular. Then $V\otimes V^c$ is also CFT-type and self-dual (and also regular). Thus it is simple. Therefore, the homomorphism of $V\otimes V^c$-modules
\begin{gather}
V\otimes V^c\rightarrow U,\qquad v\otimes v'\mapsto Y(v)_{-1}Y(v')_{-1}\id\label{eq10}
\end{gather}
(cf. \cite{LL12} proposition 3.12.7) must be injective. Thus, we can regard $V\otimes V^c$ as a \emph{conformal} sub-VOA of $U$ sharing the same conformal vector $\omega=\nu+\nu'=\nu\otimes\id+\id\otimes\nu'$. Note that by the  identification $v\otimes v'=Y(v)_{-1}Y(v')_{-1}\id$, we have $\id=\id\otimes \id,v=v\otimes\id,v'=\id\otimes v'$.

Recall that by \cite{FHL93} chapter 4, any irreducible $V\otimes V^c$-module is the tensor product of a $V$-module and a $V^c$-module. Moreover, by \cite{ADL05} theorem 2.10, any irreducible intertwining operator of $V\otimes V^c$ can be written as a sum of tensor products of irreducible intertwining operators of $V$ and of $V^c$. Therefore, any $U$-module, considered as a $V\otimes V^c$-module, is a direct sum of those of the form $W_i\otimes W_{i'}$, where $W_i$ is an irreducible $V$-module and $W_{i'}$ is an irreducible $V^c$-module. Theorem 2.10 also implies that  any intertwining operator of $U$ can be decomposed as a sum of $\mc Y_\upalpha\otimes\mc Y_{\upalpha'}$, where $\mc Y_\upalpha$ and $\mc Y_{\upalpha'}$ are irreducible intertwining operators of $V$ and $V^c$ respectively.

In the following, $W_I,W_J,W_K,\dots$ will denote $U$-modules, and $W_{i'},W_{j'},W_{k'},\dots$ will denote $V^c$-modules.  $\mc V^U{K\choose I~J}$ and $\mc V'{k'\choose i'~j'}$ will denote the corresponding vector spaces of intertwining operators of $U$ and $V^c$ respectively. Note that $W_I$ can not be regarded as a $V$-module (unless when $\omega=\nu$) but only as a weak $V$-module (see \cite{DLM97} for the definition.)

\begin{df}
Let $W_i$ be an irreducible $V$-module and $W_I$ be a $U$-module. Let $\varphi:W_i\rightarrow W_I$ and $\psi:W_I\rightarrow W_i$ be homomorphisms of weak $V$-modules, i.e., they intertwine the actions of $V$. We say that $\varphi$ is \textbf{grading-preserving} if $\varphi$ maps each $L_0$-eigenspace of $W_i$ into an $L^U_0$-eigenspace of $W_I$. We say that $\psi$ is \textbf{grading-preserving} if the preimage under $\psi$ of any $L^U_0$-eigenspace of $W_I$ is contained in an $L_0$-eigenspace of $W_i$.
\end{df}

\begin{rem}
We have seen that there is an identification of $V\otimes V^c$-modules:
\begin{align}
W_I\simeq\bigoplus_{s\in\mc E}W_s\otimes W_{\sigma(s)}\label{eq11}
\end{align}
where for each $s\in\mc E$, $W_{\sigma(s)}$ is a (non-necessarily irreducible) $V^c$-module.  We can also regard \eqref{eq11} as a decomposition of $W_I$ into irreducible weak $V$-modules, where for each $s\in\mc E$,  $W_{\sigma(s)}$ is the multiplicity space of $W_s$. (Note that different elements in $\mc E$ give rise to non-equivalent irreducible modules.) Choose any $s\in\mc E$. Choose $\varphi:W_s\rightarrow W_I$ and $\psi:W_I\rightarrow W_s$ to be  homomorphisms of weak $V$-modules. Then it is not hard to see that we can find $w^{(\sigma(s))}\in W_{\sigma(s)}$ and $\varpi\in W_{\sigma(s)}^*$ (note that $W_{\sigma(s)}^*$ is the dual vector space of $W_{\sigma(s)}$) such that
\begin{gather}
\varphi =\id_s\otimes w^{(\sigma(s))},\qquad \psi=\id_s\otimes \varpi,
\end{gather}
where $w^{(\sigma(s))}$ is considered as the linear map $\Cbb\rightarrow W_{\sigma(s)}$ sending $1$ to $w^{(\sigma(s))}$.\footnote{To see that $\varphi$ can be written in this way, choose any $t\in\mc E$ and $\upomega\in W_{\sigma(t)}^*$, and consider the  homomorphism of irreducible $V$-modules $T_\upomega:W_s\rightarrow W_t$ defined by $T_\upomega=(\id_t\otimes\upomega)\circ\varphi$. Then $T_\upomega=0$ whenever $s\neq t$ (since $W_s\nsimeq W_t$). So the image of $\varphi$ is in $W_s\otimes W_{\sigma(s)}$. Now assume $t=s$. Then $T_\upomega$ is a scalar. Choose a basis $\{e_1,e_2,\dots\}$ of $W_s$, and write $\varphi(e_1)=\sum_n e_n\otimes w_n$ where each $w_n$ is in $W_{\sigma(s)}$. Then, for each $\upomega$, $T_\upomega(e_1)=\sum_n\upomega(w_n)e_n$ is a scalar multiple of $e_1$, which shows that $w_n=0$ when $n>1$. Thus $\varphi=\id_s\otimes w_1$. That $\psi$ has the desired form can be proved similarly.} Moreover, outside $W_s\otimes W_{\sigma(s)}$, $\id_s\otimes \varpi$ is defined to be the zero functional. Thus, it is clear that \emph{$\varphi$ (resp. $\psi$)  is grading-preserving if any only if $w^{(\sigma(s))}$ (resp. $\varpi$) equals an ($L'_0$-) homogeneous vector  of $W_{\sigma(s)}$ (resp. $W_{\ovl{\sigma(s)}}$)}.
\end{rem}	

\begin{df}
Let $W_i$ be an irreducible $V$-module and $W_I$ be a $U$-module. We say that $W_i$ is a \textbf{compression} of $W_I$ if $W_I$ (considered as a weak $V$-module) has an irreducible weak $V$-submodule isomorphic to $W_i$. Equivalently, the $V\otimes V^c$-module $W_I$ has a (non-trivial) irreducible submodule isomorphic to $W_i\otimes W_{i'}$ for some irreducible $V^c$-module $W_{i'}$.
\end{df}

\begin{df}\label{lb8}
Let $\mc Y\in\mc V{k\choose i~j}$ be an irreducible intertwining operator of $V$.

(a) Let $\mc Y^U\in\mc V^U{K\choose I~J}$ be an intertwining operator of $U$. We say that $\mc Y$ is a \textbf{compression} of $\mc Y^U$, if there exit $\lambda\in\mbb Q$ and grading-preserving homomorphisms of weak $V$-modules $\varphi:W_i\rightarrow W_I$, $\phi:W_j\rightarrow W_J$, and $\psi:W_K\rightarrow W_k$, such that for any $w^{(i)}\in W_i$ and $z\in\Cbb^\times$
\begin{align*}
\mc Y(w^{(i)},z)=z^\lambda\cdot \psi\mc Y^U(\varphi w^{(i)},z)\phi.
\end{align*}

(b) If $W_I,W_J$ are $U$-modules, we say that $\mc Y$ is a \textbf{compression of type $\blt\choose I~J$ intertwining operators of $U$}, if $\mc Y$ is a (finite) sum of compressions of intertwining operators of $U$ whose charge spaces are $W_I$ and source spaces are $W_J$.

(c) If $\mc Y$ is a (finite) sum of compressions of intertwining operators of $U$, we simply say that $\mc Y$ is a \textbf{compression of  intertwining operators of $U$}
\end{df}

\begin{pp}\label{lb4}
Let $W_I,W_J,W_K$ be  $U$-modules with $V\otimes V^c$-irreducible decompositions
\begin{gather*}
W_I\simeq\bigoplus W_i\otimes W_{i'},\qquad W_J\simeq\bigoplus W_j\otimes W_{j'},\qquad W_K\simeq\bigoplus W_k\otimes W_{k'}.
\end{gather*}
Then, according to these decompositions, any $\mc Y^U\in\mc V^U{K\choose I~J}$ can be written as a sum of elements of the form $\mc Y\otimes \mc Y'$, where $\mc Y\in\mc V{k\choose i~j}$ is the compression of a type $K\choose I~J$ intertwining operator of $U$,  and $\mc Y'$ is an irreducible intertwining operator of $V^c$.
\end{pp}

\begin{proof}
We fix irreducible $V\otimes V^c$-submodules $W_i\otimes W_{i'},W_j\otimes W_{j'},W_k\otimes W_{k'}$ of $W_I,W_J,W_k$ respectively. Let $\Theta^{k'}_{i',j'}$ be a basis of $\Hom_{V^c}(W_{i'}\boxtimes W_{j'},W_{k'})$. We have a (functorial) isomorphism
\begin{align}
\Hom_{V^c}(W_{i'}\boxtimes W_{j'},W_{k'})\xrightarrow{\simeq} \mc V'{k'\choose i'~j'},\qquad \upalpha'\mapsto\mc Y'_{\upalpha'}\label{eq21}
\end{align} 
similar to \eqref{eq12}. Consider $\mc Y^U$ as an intertwining operator of $V\otimes V^c$, and restrict it to $W_i\otimes W_{i'},W_j\otimes W_{j'},W_k\otimes W_{k'}$. Then, by \cite{ADL05} theorem 2.10, this restriction is a sum of tensor products of $V$- and $V^c$-intertwining operators. Assume without loss of generality that this restriction is non-zero. Then $W_k,W_{k'}$ must be irreducible submodules of $W_i\boxtimes W_j,W_{i'}\boxtimes W_{j'}$ respectively. Now, for each $\upalpha'\in\Theta^{k'}_{i',j'}$, we can find $\upalpha\in\Hom_V(W_i\boxtimes W_j,W_k)$ (not necessarily in $\Xi^k_{i,j}$) such that the restriction of $\mc Y^U$   equals
\begin{align*}
\sum_{\upalpha'\in\Theta^{k'}_{i',j'}}\mc Y_\upalpha\otimes\mc Y'_{\upalpha'}.
\end{align*}
We shall show that each $\mc Y_\upalpha$ is a sum of  compressions of  type $K\choose I~J$ intertwining operators of $U$.

Choose $\upalpha'\in\Theta^{k'}_{i',j'}$ and apply corollary \ref{lb2} to $V^c$. Then there exist homogeneous vectors $w^{(i')}_1,\dots,w^{(i')}_m\in W_{i'}$,  $w^{(j')}_1,\dots,w^{(j')}_m\in W_{j'}$, $w^{(\ovl {k'})}\in W_{\ovl {k'}}$, and constants $\lambda_1,\dots,\lambda_m\in\mbb Q$, such that for any $\upbeta'\in\Theta^{k'}_{i',j'}$, the expression
\begin{align}
\sum_{l=1}^m z^{\lambda_l}\bk{\mc Y'_{\upbeta'}(w^{(i')}_l,z)w^{(j')}_l,w^{(\ovl {k'})}}\label{eq13}
\end{align}
is a constant (over the complex variable $z$), and this constant is non-zero if and only if $\upbeta'=\upalpha'$. By scaling the vector $w^{(\ovl{k'})}$, we may assume that when $\upbeta'=\upalpha'$, the above constant is $1$. Now, for each $l=1,2,\dots,m$, define grading-preserving homomorphisms of weak $V$-modules $\varphi_l:W_i\rightarrow W_I,\phi_l:W_j\rightarrow W_J,\psi:W_K\rightarrow W_k$ by
\begin{align*}
\varphi_l=\id_i\otimes w_l^{(i')},\qquad \phi_l=\id_j\otimes w_l^{(j')},\qquad \psi=\id_k\otimes w^{(\ovl{k'})}.
\end{align*}
Then we have
\begin{align*}
\mc Y_\upalpha(w^{(i)},z)=\sum_{l=1}^mz^{\lambda_l}\cdot \psi\mc Y^U(\varphi_lw^{(i)},z)\phi_l.
\end{align*}
This finishes the proof.
\end{proof}

\section{Proof of the main result}

In this section, we assume that $V^{cc}=V$. Let $W_I,W_J$ be $V$-modules, and we fix irreducible decompositions of $V\otimes V^c$-modules:
\begin{gather}
U\simeq \big(V\otimes V^c\big)\oplus\Big(\bigoplus W_a\otimes W_{a'}\Big),\label{eq17}\\
W_I\simeq \big(V\otimes V^c\big)\oplus\Big(\bigoplus W_i\otimes W_{i'}\Big),\\
W_J\simeq \big(V\otimes V^c\big)\oplus\Big(\bigoplus W_j\otimes W_{j'}\Big).
\end{gather}
We first recall the following obvious fact.

\begin{pp}
If $W_a\otimes W_{a'}$ is an irreducible $V\otimes V^c$-submodule of $U\ominus (V\otimes V^c)$, then $W_a$ is not isomorphic to $V$ and $W_{a'}$ is not isomorphic to $V^c$.\label{lb5}
\end{pp}

\begin{proof}
If $W_{a'}$ is isomorphic to $V^c$, then $W_{a'}$ contains a non-zero homogeneous vector $w_2$ equivalent to the vacuum vector of $V^c$. Choose any non-zero $w_1\in W_a$. Then for any $v'\in V^c$ and $n\in\Nbb$, $Y'(v')_nw_2=0$. Therefore $Y^U(v')_n(w_1\otimes w_2)=w_1\otimes Y'(v')_nw_2=0$. Thus $w_1\otimes w_2\in V^{cc}=V=V\otimes\id$, which is impossible since $w_1\otimes w_2$ is not in $V\otimes V^c$. So $W_{a'}$ is not isomorphic to $V^c$. Since $V^{ccc}=V^c$, for a similar reason, $W_a$ is also not isomorphic to $V$.
\end{proof}

For each irreducible $V^c$-module, we choose a representing element, and let them form a finite set $\mc E'$. Assume $W_{0'}:=V^c$ is in $\mc E'$. If $W_{i'},W_{j'},W_{k'}$ are  $V^c$-modules, we choose a basis $\Theta^k_{i,j}$ of $\Hom_{V^c}(W_{i'}\boxtimes W_{j'},W_{k'})$. The linear isomorphism
\begin{gather}
\tth:\Hom_{V^c}(W_{i'}\boxtimes W_{j'},W_{k'})^*\xrightarrow\simeq \Hom_{V^c}(W_{\ovl {i'}}\boxtimes W_{k'},W_{j'})
\end{gather}
and the morphism
\begin{gather*}
\ev_{\ovl{i'},i'}\in\Hom_{V^c}(W_{\ovl{i'}}\boxtimes W_{i'},V^c)
\end{gather*}
are defined as in section \ref{lb3}. Then $\mc Y'_{\ev_{\ovl{i'},i'}}$ is the annihilation operator of $W_{i'}$.

According to the decomposition for $W_I$, $W_{\ovl I}$ also has the corresponding decomposition:
\begin{gather*}
W_{\ovl I}\simeq \big(V\otimes V^c\big)\oplus\Big(\bigoplus W_{\ovl i}\otimes W_{\ovl{i'}}\Big).
\end{gather*}
Let
\begin{gather*}
\ev_{\ovl I,I}\in\Hom_U(W_{\ovl I}\boxtimes W_I,U)
\end{gather*}
which corresponds to the annihilation operator $\mc Y^U_{\ev_{\ovl I,I}}$ of the $U$-module $W_I$. Suppose that in the above decompositions, $W_i\boxtimes W_{i'}$ is an irreducible submodule of $W_I\ominus(V\otimes V^c)$. If we regard $\mc Y^U_{\ev_{\ovl I,I}}$ as an intertwining operator of $V\otimes V^c$, then it is easy to see that the restriction of $\mc Y^U_{\ev_{\ovl I,I}}$ to the charge subspace $W_{\ovl i}\boxtimes W_{\ovl{i'}}$ source subspace $W_i\boxtimes W_{i'}$ and the target subspace $V\otimes V^c$ is $\mc Y_{\ev_{\ovl i,i}}\otimes \mc Y'_{\ev_{\ovl{i'},i'}}$, the tensor product of the annihilation operators of $W_i$ and of $W_{i'}$.

\begin{thm}
Let $V$ be a vertex operator subalgebra of $U$ satisfying \eqref{eq16}. Assume that $U$, $V$, and  $V^c$ are CFT type, self-dual, and regular VOAs. Assume also that $V^{cc}=V$. Let $W_I,W_J$ be $U$-modules. Let $W_i,W_j$ be irreducible $V$-modules that are compressions of $W_I$ and $W_J$ respectively. Then any irreducible intertwining operator of $V$ with charge space $W_i$ and source space $W_j$ is a compression of type $\blt\choose I~J$ intertwining operators of $U$.
\end{thm}

\begin{proof}
Let $W_i\otimes W_{i'}$ and $W_j\otimes W_{j'}$ be irreducible $V\otimes V^c$-submodules of $W_I,W_J$ respectively.  Assume that $k\in\mc E$ and not all type $k\choose i~j$ intertwining operators of $V$ are compressions of type $\blt\choose I~J$ intertwining operators of $U$. Let $\scr V{k\choose i~j}$ be a subspace of $\mc V{k\choose i~j}$ with codimension $1$ containing all elements of $\mc V{k\choose i~j}$ that are compressions of type $\blt\choose I~J$ intertwining operators of $U$. Choose a nonzero element $\fk A\in\Hom_V(W_i\boxtimes W_j,W_k)$ such that $\mc Y_{\fk A}\notin\scr V{k\choose i~j}$. We assume that the basis $\Xi^k_{i,j}$ of $\Hom_V(W_i\boxtimes W_j,W_k)$ is chosen such that $\fk A\in\Xi^k_{i,j}$, and that $\mc Y_\upalpha\in\scr V{k\choose i~j}$ for any $\upalpha\in\Xi^k_{i,j}$ not equal to $\fk A$. 

Choose $z,\zeta\in\Cbb$ satisfying \eqref{eq8}. Recall that  $Y^U_I$ is the $U$-vertex operator of $W_I$ and $\mc Y^U_{\ev_{\ovl I,I}}$ is the $U$-annihilation operator of $W_I$. In the following, we shall calculate the fusion relation for the iterate of $V\otimes V^c$-intertwining operators $Y^U_J$ and $\mc Y^U_{\ev_{\ovl I,I}}$ (with restricted charge, source, and target spaces) in two ways. These two methods will give incompatible results, which therefore lead to a contradiction. Let $\uppi_{j\otimes j'}$ be the projection of the algebraic completion of $W_J$ onto the one of $W_j\otimes W_{j'}$.

Step 1. Note that for each $s\in\mc E,s'\in\mc E'$, the set $\Xi^s_{i,j}\times \Theta^{s'}_{i',j'}$ (more precisely, $\{\mc Y_\upalpha\otimes\mc Y'_{\upalpha'}:\upalpha\in \Xi^s_{i,j},\upalpha'\in\Theta^{s'}_{i',j'} \}$) is a basis of the vector space of type $W_s\otimes W_{s'}\choose W_i\otimes W_{i'}~W_j\otimes W_{j'}$ intertwining operators of $V\otimes V^c$. (See \cite{ADL05} theorem 2.10; it is also an easy consequence of corollary \ref{lb2}.) Thus, for any $s\in\mc E,s'\in\mc E'$ and $\upalpha,\upbeta\in\Hom_V(W_i\boxtimes W_j,W_s),\upalpha',\upbeta'\in\Hom_{V^c}(W_{i'}\boxtimes W_{j'},W_{s'})$, there is a \emph{unique} constant $\lambda_{\upalpha,\upbeta,\upalpha',\upbeta'}\in\Cbb$ such that for any
\begin{align}
w_1\in W_{\ovl i},w_2\in W_{\ovl{i'}},w_3\in W_i,w_4\in W_{i'},w_5\in W_j,w_6\in W_{j'},\label{eq18}
\end{align}
the following fusion relation of $V\otimes V^c$-intertwining operators holds:
\begin{align}
&\uppi_{j\otimes j'}\cdot Y^U_J\big(\mc Y^U_{\ev_{\ovl I,I}}(w_1\otimes w_2,z-\zeta)(w_3\otimes w_4),\zeta\big)(w_5\otimes w_6)\nonumber\\
=&\sum_{\substack{s\in\mc E\\s'\in\mc E'}}\sum_{\substack{\upalpha,\upbeta\in\Xi^s_{i,j}\\\upalpha',\upbeta'\in\Theta^{s'}_{i',j'}}}\lambda_{\upalpha,\upbeta,\upalpha',\upbeta'}\cdot \mc Y_{\tth(\wch\upbeta)}(w_1,z)\mc Y_\upalpha(w_3,\zeta)w_5\otimes \mc Y'_{\tth(\wch{\upbeta'})}(w_2,z)\mc Y'_{\upalpha'}(w_4,\zeta)w_6.\label{eq19}
\end{align}
(Recall \eqref{eq12} and \eqref{eq21} for the notations $\mc Y,\mc Y'$.) On the other hand, by \eqref{eq9}, the iterate of the $U$-intertwining operators $Y^U_J$ and $\mc Y^U_{\ev_{\ovl I,I}}$ equals a sum of products of type $J\choose \ovl I~\blt$ intertwining operators and type $\blt\choose I~J$ intertwining operators of $U$. Therefore, by proposition \ref{lb4} and the uniqueness of fusion coefficients, we have
\begin{align}
\lambda_{\fk A,\upbeta,\upalpha',\upbeta'}=0\label{eq27}
\end{align}
for any $s'\in\mc E'$, $\upbeta\in\Xi^k_{i,j}$, and $\upalpha',\upbeta'\in\Theta^{s'}_{i',j'}$. In particular, the right hand side of \eqref{eq19} has no terms containing $\mc Y_{\tth(\wch{\fk A})}(w_1,z)\mc Y_{\fk A}(w_3,\zeta)w_5$.

Step 2. We calculate the iterate of $\uppi_{j\otimes j'}\cdot Y^U_J$ and $\mc Y^U_{\ev_{\ovl I,I}}$ using a different method, and show that some terms containing $\mc Y_{\tth(\wch{\fk A})}(w_1,z)\mc Y_{\fk A}(w_3,\zeta)w_5$ will appear. By the paragraph before the theorem, we know that for any $w_1,w_2,\dots,w_6$ as in \eqref{eq18},
\begin{align}
&\mc Y^U_{\ev_{\ovl I,I}}(w_1\otimes w_2,z-\zeta)(w_3\otimes w_4)\nonumber\\
=&\mc Y_{\ev_{\ovl i,i}}(w_1,z-\zeta)w_3\otimes \mc Y'_{\ev_{\ovl{i'},i'}}(w_2,z-\zeta)w_4\nonumber\\
&+\sum_{W_a\otimes W_{a'}}\sum_{\upgamma,\upgamma'} \mc Y_\upgamma(w_1,z-\zeta)w_3\otimes \mc Y'_{\upgamma'}(w_2,z-\zeta)w_4\label{eq20}
\end{align}
where the first sum is over all irreducible $V\otimes V^c$-submodules of $U\ominus (V\otimes V^c)$ as in the decomposition \eqref{eq17}, $\upgamma\in\Hom_V(W_{\ovl i}\boxtimes W_i,W_a)$, and $\upgamma'\in\Hom_{V^c}(W_{\ovl{i'}}\boxtimes W_{i'},W_{a'})$. We shall now calculate the iterate of $\uppi_{j\otimes j'}\cdot Y^U_J$ with each term on the right hand side of \eqref{eq20}.

The first term is in the algebraic completion of $V\otimes V^c$. Moreover, the restriction of  $Y^U_J$  (regarded as a $V\otimes V^c$-intertwining operator) to  $V\otimes V^c,W_j\otimes W_{j'},W_j\otimes W_{j'}$ equals $Y_j\otimes Y'_{j'}$, where $Y_j,Y'_{j'}$ are respectively the vertex operators of the $V$-module $W_j$ and the $V^c$-module $W_{j'}$. Therefore, by \eqref{eq9},
\begin{align}
&\uppi_{j\otimes j'}\cdot Y^U_J\big(\mc Y_{\ev_{\ovl i,i}}(w_1,z-\zeta)w_3\otimes \mc Y'_{\ev_{\ovl{i'},i'}}(w_2,z-\zeta)w_4,\zeta\big)(w_5\otimes w_6)\nonumber\\
=&Y_j\big(\mc Y_{\ev_{\ovl i,i}}(w_1,z-\zeta)w_3,\zeta \big)w_5\otimes Y'_{j'}\big(\mc Y'_{\ev_{\ovl{i'},i'}}(w_2,z-\zeta)w_4,\zeta \big)w_6\nonumber \\
=&\sum_{s\in\mc E}\sum_{\upalpha\in\Xi^s_{i,j}}\mc Y_{\tth(\wch\upalpha)}(w_1,z)\mc Y_\upalpha(w_3,\zeta)w_5\otimes Y'_{j'}\big(\mc Y'_{\ev_{\ovl{i'},i'}}(w_2,z-\zeta)w_4,\zeta \big)w_6.\label{eq25}
\end{align}
In the above expression, (the sum of) all the terms containing $\mc Y_{\tth(\wch{\fk A})}(w_1,z)\mc Y_{\fk A}(w_3,\zeta)w_5$ is
\begin{align}
\mc Y_{\tth(\wch{\fk A})}(w_1,z)\mc Y_{\fk A}(w_3,\zeta)w_5\otimes Y'_{j'}\big(\mc Y'_{\ev_{\ovl{i'},i'}}(w_2,z-\zeta)w_4,\zeta \big)w_6.\label{eq23}
\end{align}

On the other hand, suppose that when restricted to the charge subspace $W_a\otimes W_{a'}$ (where $W_a\otimes W_{a'}$ is an irreducible submodule of $U\ominus(V\otimes V^c)$) and source and target subspace $W_j\otimes W_{j'}$, the $V\otimes V^c$-intertwining operator $Y^U_J$ could be written as $\sum_{\updelta,\updelta'}\mc Y_\updelta\otimes\mc Y'_{\updelta'}$, where each $\mc Y_\updelta$ is of type $j\choose a~j$ and $\mc Y'_{\updelta'}$ has type $j'\choose a'~j'$. Then the iterate of $\uppi_{j\otimes j'}\cdot Y^U_J$ with the second  term of \eqref{eq20} is
\begin{align}
&\sum_{W_a\otimes W_{a'}}\sum_{\upgamma,\upgamma'}\uppi_{j\otimes j'}\cdot Y^U_J\big(\mc Y_\upgamma(w_1,z-\zeta)w_3\otimes \mc Y_{\upgamma'}(w_2,z-\zeta)w_4,\zeta  \big)(w_5\otimes w_6)\nonumber\\
=&\sum_{W_a\otimes W_{a'}}\sum_{\substack{\upgamma,\upgamma'\\\updelta,\updelta'}}\mc Y_\updelta\big(\mc Y_\upgamma(w_1,z-\zeta)w_3,\zeta \big)w_5\otimes \mc Y'_{\updelta'}\big(\mc Y'_{\upgamma'}(w_2,z-\zeta)w_4,\zeta \big)w_6.\label{eq26}
\end{align}
If we write each $\mc Y_\updelta\big(\mc Y_\upgamma(w_1,z-\zeta)w_3,\zeta \big)w_5$ as a sum of products of $V$-intertwining operators under the bases $\Xi^s_{i,j}$ and $\{\tth(\wch{\upalpha}):\upalpha\in\Xi^s_{i,j}\}$ (over all $s\in\mc E$) similar to part of \eqref{eq19}, then the sum of all the  terms containing $\mc Y_{\tth(\wch{\fk A})}(w_1,z)\mc Y_{\fk A}(w_3,\zeta)w_5$ should be
\begin{align}
\mc Y_{\tth(\wch{\fk A})}(w_1,z)\mc Y_{\fk A}(w_3,\zeta)w_5\otimes \sum_{W_a\otimes W_{a'}}\sum_{\substack{\upgamma,\upgamma'\\\updelta,\updelta'}}\kappa_{\upgamma,\updelta}\cdot \mc Y'_{\updelta'}\big(\mc Y'_{\upgamma'}(w_2,z-\zeta)w_4,\zeta \big)w_6\label{eq24}
\end{align}
where each $\kappa_{\upgamma,\updelta}$ is a constant. By proposition \ref{lb5}, every $W_{a'}$ (which is irreducible) is not isomorphic to $V^c$. Therefore,  as the linear map $\Phi_{z,\zeta}$ (see \eqref{eq22}) is injective,  the  sum of \eqref{eq23} and \eqref{eq24} is not zero for some $w_1,\dots,w_6$ satisfying \eqref{eq18}. This shows that \eqref{eq19} (which is the sum of \eqref{eq25} and \eqref{eq26}) has non-zero terms containing $\mc Y_{\tth(\wch{\fk A})}(w_1,z)\mc Y_{\fk A}(w_3,\zeta)w_5$. In other words, $\lambda_{\fk A,\fk A,\upalpha',\upbeta'}\neq 0$ for some $s'\in\mc E'$ and $\upalpha',\upbeta'\in\Theta^{s'}_{i',j'}$. This gives a contradiction.
\end{proof}

The following result was proved in \cite{KM15}:

\begin{thm}\label{lb6}
Let $V$ be a vertex operator subalgebra of $U$ satisfying \eqref{eq16}. Assume that $U$, $V$, and  $V^c$ are CFT type, self-dual, and regular VOAs. Assume also that $V^{cc}=V$. Then any irreducible $V$-module is the compression of a $U$-module.
\end{thm}

The above two theorems imply immediately the following:

\begin{thm}\label{lb7}
Under the assumption of theorem \ref{lb6}, any irreducible intertwining operator of $V$ is a compression of intertwining operators of $U$.
\end{thm}

\noindent {\small \sc Department of Mathematics, Rutgers University, USA.}

\noindent {\em E-mail}: bin.gui@rutgers.edu\qquad binguimath@gmail.com

\begin{thebibliography}{999999}
	\footnotesize	
\bibitem[ABD04]{ABD04}
Abe, T., Buhl, G. and Dong, C., 2004. Rationality, regularity, and $C_2$-cofiniteness. Transactions of the American Mathematical Society, 356(8), pp.3391-3402.	

\bibitem[ACL19]{ACL19}
Arakawa, T., Creutzig, T. and Linshaw, A.R., 2019. W-algebras as coset vertex algebras. Inventiones mathematicae, 218(1), pp.145-195.

\bibitem[ADL05]{ADL05}
Abe, T., Dong, C.,  Li, H. 2005.  Fusion Rules for the Vertex Operator Algebras $M(1)^+$ and $V_L^+$. Comm. Math. Phys, 253, pp.171-219.

\bibitem[AM88]{AM88}
Anderson, G. and Moore, G., 1988. Rationality in conformal field theory. Communications in mathematical physics, 117(3), pp.441-450.

\bibitem[Ara15a]{Ara15a}
Arakawa, T., 2015. Associated varieties of modules over Kac–Moody algebras and $C_2$-cofiniteness of $W$-algebras. International Mathematics Research Notices, 2015(22), pp.11605-11666.

\bibitem[Ara15b]{Ara15b}
Arakawa, T., 2015. Rationality of W-algebras: principal nilpotent cases. Annals of Mathematics, pp.565-604.


\bibitem[CKLW18]{CKLW18}
Carpi, S., Kawahigashi, Y., Longo, R. and Weiner, M., 2018. From vertex operator algebras to conformal nets and back (Vol. 254, No. 1213). Memoir of the American Mathematical Society.

\bibitem[DLM97]{DLM97}
Dong, C., Li, H. and Mason, G., 1997. Regularity of rational vertex operator algebras. Advances in Mathematics, 132(1), pp.148-166.

\bibitem[DLM00]{DLM00}
Dong, C., Li, H. and Mason, G., 2000. Modular-Invariance of Trace Functions in Orbifold Theory and Generalized Moonshine. Communications in Mathematical Physics, 214(1), pp.1-56.

\bibitem[FHL93]{FHL93}
Frenkel, I., Huang, Y.Z. and Lepowsky, J., 1993. On axiomatic approaches to vertex operator algebras and modules (Vol. 494). American Mathematical Soc..

\bibitem[FZ92]{FZ92}
Frenkel, I.B. and Zhu, Y., 1992. Vertex operator algebras associated to representations of affine and Virasoro algebras. Duke Mathematical Journal, 66(1), pp.123-168.

	\bibitem[Gui18]{Gui18}
	Gui, B., 2018. Categorical extensions of conformal nets. arXiv preprint arXiv:1812.04470.
	
	\bibitem[Gui19a]{Gui19a}
	Gui, B., 2019. Unitarity of the modular tensor categories associated to unitary vertex operator algebras, I,  Comm. Math. Phys., 366(1), pp.333-396. 
	

	
	\bibitem[Gui19b]{Gui19b}
	Gui, B., 2019. Energy bounds condition for intertwining operators of type $ B $, $ C $, and $ G_2 $ unitary affine vertex operator algebras. Trans. Amer. Math. Soc. 372 (2019), 7371-7424
	
	\bibitem[Gui20]{Gui20}
	Gui, B., 2020. Unbounded field operators in categorical extensions of conformal nets. arXiv preprint arXiv:2001.03095.
	


\bibitem[HK07]{HK07}
Huang, Y.Z. and Kong, L., 2007. Full field algebras. Communications in mathematical physics, 272(2), pp.345-396.

\bibitem[HL13]{HL13}
Huang, Y.Z. and Lepowsky, J., 2013. Tensor categories and the mathematics of rational and logarithmic conformal field theory. Journal of Physics A: Mathematical and Theoretical, 46(49), p.494009.

\bibitem[Hua95]{Hua95}
Huang, Y.Z., 1995. A theory of tensor products for module categories for a vertex operator algebra, IV. Journal of Pure and Applied Algebra, 100(1-3), pp.173-216.

\bibitem[Hua05]{Hua05}
Huang, Y.Z., 2005. Differential equations and intertwining operators. Communications in Contemporary Mathematics, 7(03), pp.375-400.

\bibitem[Hua08]{Hua08}
Huang, Y.Z., 2008. Rigidity and modularity of vertex tensor categories. Communications in contemporary mathematics, 10(supp01), pp.871-911.

\bibitem[KM15]{KM15}
Krauel, M. and Miyamoto, M., 2015. A modular invariance property of multivariable trace functions for regular vertex operator algebras. Journal of Algebra, 444, pp.124-142.

\bibitem[LL12]{LL12}
Lepowsky, J. and Li, H., 2012. Introduction to vertex operator algebras and their representations (Vol. 227). Springer Science \& Business Media.

\bibitem[Ten17]{Ten17}
Tener, J.E., 2017. Construction of the unitary free fermion Segal CFT. Communications in Mathematical Physics, 355(2), pp.463-518.

\bibitem[Ten19a]{Ten19a}
Tener, J.E., 2019. Geometric realization of algebraic conformal field theories. Advances in Mathematics, 349, pp.488-563.

\bibitem[Ten19b]{Ten19b}
Tener, J.E., 2019. Representation theory in chiral conformal field theory: from fields to observables. Selecta Mathematica, 25(5), p.76.

\bibitem[Ten19c]{Ten19c}
Tener, J.E., 2019. Fusion and positivity in chiral conformal field theory. arXiv preprint arXiv:1910.08257.

\bibitem[TL04]{TL04}
Toledano-Laredo, V., 2004. Fusion of positive energy representations of $\mathrm{LSpin}_{2n}$. arXiv preprint math/0409044.

\bibitem[Was98]{Was98}
Wassermann, A., 1998. Operator algebras and conformal field theory III. Fusion of positive energy representations of $LSU(N)$ using bounded operators. Inventiones mathematicae, 133(3), pp.467-538.
	
	
\end{thebibliography}
\end{document}